\documentclass[12pt]{amsart}

\usepackage{amsmath,amssymb}
\usepackage{graphics}
\usepackage{enumerate}
\usepackage[dvips]{color}
\usepackage{version}
\usepackage{stmaryrd}
\usepackage[all]{xy}

\usepackage{todonotes}
\newtheorem{Thm}{Theorem}[section] %%%%%%%%%%%%%%%%%%%%%
\newtheorem{Lem}[Thm]{Lemma}

\newtheorem{Cor}[Thm]{Corollary}
\newtheorem{Cor.Conj}[Thm]{Corollary of Conjecture}

\newtheorem{Prop}[Thm]{Proposition}

\theoremstyle{remark}
\newtheorem{Rem}[Thm]{Remark}
\newtheorem{Ex}[Thm]{Example}

\theoremstyle{definition}
\newtheorem{Def}[Thm]{Definition}

\newtheorem*{ack}{Acknowledgements}

\newcommand{\R}{\ensuremath{\mathbb{R}}}
\newcommand{\C}{\ensuremath{\mathbb{C}}}

\newcommand{\Z}{\ensuremath{\mathbb{Z}}}

\newcommand{\Q}{\mathbb{Q}}

\newcommand{\A}{\mathbb{A}}

\newcommand{\X}{\mathcal{X}}
\newcommand{\Y}{\mathcal{Y}}

\newcommand{\PP}{\mathbb{P}}

\begin{document}

\title[
(weak)K-moduli for K-trivial case]{On log minimality of 
weak K-moduli compactifications 
of Calabi-Yau varieties}

\author{Yuji Odaka}
\date{\today}

\maketitle

\begin{abstract}
For moduli of polarized smooth 
K-trivial a.k.a., Calabi-Yau varieties in a general sense, 
we revisit a classical problem of 
constructing its ``weak K-moduli'' 
compactifications which 
parametrizes K-{\it semi}stable (i.e., semi-log-canonical K-trivial) 
degenerations. 
Although weak K-moduli is not unique in general, 
they always contain a unique partial compactification (K-moduli). 

Our main theorem is the 
{\it log minimality} of their normalizations, 
under some conditions. Partially to confirm that known examples 
satisfy the conditions, 
we also include an appendix on the algebro-geometric 
reconstruction of Kulikov models via the MMP, 
which has been folklore at least but we somewhat strengthen. 
\end{abstract}

\section{Introduction}

In this paper, we focus on 
K-trivial variety (often called Calabi-Yau variety 
in e.g. K\"ahler geometric or birational geometric context) which 
means 
projective variety whose singularity is mild so that the 
canonical divisor makes sense and  
is linearly equivalent to $0$. 
The problem of compactifying moduli, 
say $M^{o}\subset \overline{M}$, of 
K-trivial varieties is classical and rich topic with many 
important connections with other fields. The renowned existence theorem of 
Ricci-flat K\"ahler metrics \cite{Yau} and important 
roles they play in the context of string theory 
give one such aspect. 
Rather than expanding such broad but widely wellknown backgrounds somewhat, 
we refer to our previous review \cite[\S 1.2]{galaxy} or the references cited below. 
Also, some connection with K\"ahler geometry via a compactification of moduli is explored in 
\cite{OO} but note that the compactification therein is not even a variety.  Nevertheless, we expect it to parametrize the collapsing of 
Ricci-flat metrics. 

In a purely algebro-geometric side, \cite[\S 1.2]{galaxy}, 
we introduced the terminology of 
{\it weak K-moduli} compactification for 
K-trivial varieties (see Definition \ref{weak.Kmoduli}), 
which roughly means 
compactification $\overline{M}$ 
of moduli $M^{o}$ of polarized smooth Calabi-Yau varieties 
whose boundary 
parametrizes semi-log-canonical K-trivial degenerations. 
The idea had been implicitly existed in the field, as indeed the problem of constructing such compactifiation has been 
well-pursued by many experts and hence classical topic. 
Therefore this may be regarded as a matter of words. The 
terminology ``{\it K-moduli}'' comes from the notion of 
{\it K-stability} (\cite{Tia, Don}) 
which is the algebro-geometric counterpart of existence of 
constant scalar curvature K\"ahler metrics, and is 
introduced in \cite{Od.oldsurvey} with its existence speculation, 
after various 
background works of both algebraic geometry and differential geometry. However 
we need reformulation and more precise version 
of its existence conjecture 
in general. For anticanonically polarized 
$\Q$-Fano varieties case, 
the precise 
reformulation is introduced in 
\cite[3.13, 6.2]{OSS} after \cite[1.3.1]{Spo} 
and the existence 
conjecture is now solved as a 
combination of (at least) a dozen of papers. 

In turn, for the K-trivial/Calabi-Yau case, 
K-semistability (resp., K-stability) 
is simply equivalent to semi-log-canonicity (resp., log terminality) 
\cite{Od0, Od}. 
A very important caution here is that, even for a fixed 
moduli $M^{o}$, weak K-moduli compactifications are 
{\it not} unique (e.g., \cite{Shah, AET}). 
Recall that sometimes (normalization of) weak K-moduli is shown or 
expected to be (one of) the so-called toroidal 
compactifications \cite{AMRT} (cf., \cite[\S3, \S4]{Mum77}), 
and note that they are log minimal with appropriate 
natural boundary divisor. Some other examples 
of weak K-moduli compactification 
such as \cite{AET} is normalized to be 
semi-toric compactification 
by Looijenga \cite{Looi, Looi2} which are similarly log minimal. 

One aim of this paper is to show such 
{\it log minimality} of weak K-moduli 
under certain condition as a general theorem. 

\begin{Thm}[Log minimality - see Theorem \ref{log.min} for details]
Under certain assumptions, 
the normalization of weak K-moduli compactification of 
moduli spaces of polarized K-trivial varieties, is 
log minimal. 
\end{Thm}

The point is that this log minimality gives a fairly strong restriction on the possible compactifications, 
which we plan to explore in the next work. 
For the proof, we use variation of 
mixed Hodge structures or smooth mixed Hodge modules, 
as well as recent refinement of cone theorems \cite{Sva, Fjn.hyp}. 
The result itself also matches to concrete examples of 
\cite{Namikawa, Shah, Looi2, AN, Nakamura, Alexeev, Zhu, AET, AE} 
(see also \cite{AB, ABE, HLL}) 
among others, 
but a very different point is that our proof in this notes do {\it not} 
rely on any Torelli type theorem nor the structure of Shimura varieties 
but rather of more general Hodge theoretic flavor. 
\vspace{3mm}

As an auxiliary result, we also write a proof that, 
although weak K-moduli $\overline{M}$ 
is not unique for fixed $M^{o}$, it must 
contain a particular canonical (unique) partial compactification $M$, 
which coincides with the Weil-Petersson metric 
completion of $M^{o}$, 
and is quasi-projective. 
Actually this follows from a simple combination of 
known results \cite{Vie, Wan97, DS, Tos.WP, YZha} and 
\cite{Birkar20} 
and is essentially their corollary. 
Last missing piece for a while had been 
a boundedness result (although doubted as \cite[1.2]{YZha}) 
which is now a 
big result of Birkar \cite[1.6]{Birkar20}. 
We acknowledge appreciation to Chenyang Xu for pointing it to me. 
Further, the set of parametrized objects in $M$ are 
characterised by (strict) K-stability or the minimization of the 
degree of CM line bundles as we recall in 
Lemma \ref{CM.min}, Corollary \ref{CM.min2} (also see the references). 

For the case of polarized 
K3 surfaces, this partial compactification 
is explicit and coincides with an orthogonal modular variety. 
This follows from the semi-classical fact that 
allowing polarized ADE singular K3 surfaces (or 
equivalently, quasi-polariezd smooth K3 surfaces) 
fills Heegner divisors. 
The same statements also generalizes to the case of 
irreducible symplectic varieties, 
once we replace ADE singularities by 
symplectic singularities in the sense of Beauville, 
as we showed in \cite[\S 8.3, Theorem 8.3, 
Corollary 8.4]{OO}. See {\it loc.cit} for details. 

In this paper, after clarifying the 
existence of K-moduli and basic 
properties of weak K-moduli in \S \ref{Kmod.sec}, 
in \S \ref{wKmod.sec} we discuss 
weak K-moduli and its log minimality. 
The proof depends on analysis of variation of mixed Hodge 
structures in \S 4. Finally, to fill the lack of reference, we 
discuss an algebraic reconstruction of Kulikov models for 
K-trivial surfaces which is indeed substantially related to the 
weak K-moduli problem for K-trivial surfaces. 

Except for appendix, 
we work over an arbitray algebraically closed field $k$ 
of characteristics $0$ but \S 3, \S 4 requires $k=\C$ since 
we use Hodge theory. 

\section{Preparation - Setup and K-moduli}\label{Kmod.sec}

We make a general setup 
by fixing a connected Deligne-Mumford 
moduli stack $\mathcal{M}^{o}$ 
of polarized smooth K-trivial varieties 
i.e., there is a universal family 
$\pi^{o}\colon (\mathcal{U}^{o},\mathcal{L}^{o})\to 
\mathcal{M}^{o}$. We define weak K-moduli compactification 
and also K-moduli (partial compactification) below. 

\begin{Def}[Weak K-moduli cf., \cite{galaxy}]\label{weak.Kmoduli}
We follow the above notation. We call a proper Deligne-Mumford moduli stack $\overline{\mathcal{M}}$ compactifying $\mathcal{M}$, 
together with a $\Q$-Gorenstein family of 
polarized semi-log-canonical, or equivalently, K-semistable\footnote{the equivalence follows  
from \cite{Od, Od0}}) 
Calabi-Yau varieties 
$\bar{\pi}\colon (\bar{\mathcal{U}},\bar{\mathcal{L}})\to \overline{\mathcal{M}}$, a 
{\it weak K-moduli stack} if it satisfies the following two 
conditions: 
\begin{enumerate}
\item $\bar{\pi}$
extends 
$\pi\colon (\mathcal{U}^{o},\mathcal{L}^{o})\to \mathcal{M}^{o}$,
\item \label{effectivity}
underlying family of varieties 
$\bar{\mathcal{U}}\to \overline{\mathcal{M}}$ 
 is effective (i.e., no isomorphic varieties occur as fibers at 
 different $k$-rational points).  
\end{enumerate}
\end{Def}

We note that when you apply the theory of proper moduli of 
``stable pairs'', or equivalently log (i.e., attaching divisors) 
KSBA theory, the second assertion 
\eqref{effectivity} is a priori quite nontrivial if it holds. Indeed, 
the theory a priori only gives a moduli of {\it log pairs} 
(encoding the additional information of $\Q$- or $\R$-divisors) 
so that there may be locus in which the underlying varieties 
do not deform and only divisors deform. 
Nevertheless, if one takes a general section of very ample 
$(\overline{L}\otimes \overline{\pi}^{*}N)^{\otimes a}$ 
for $N\gg 0$ and ample $N$ on $\overline{M}$, 
it immediately follows that in general 
weak K-moduli should be an algebraic substack of some 
``log KSBA''-type moduli stack (if it exists) by 
taking a general element of relatively very ample 
linear system on the universal family. 
See \cite[Section \S 1.2]{galaxy} for more detailed discussions. 

The following is announced in \cite[Remark 9.1]{OO}. 

\begin{Thm}[K-moduli]\label{Kmod}
For fixed moduli algebraic stack $\mathcal{M}^{o}$ of 
polarized smooth K-trivial varieties as above, we have a 
finite type Deligne-Mumford algebraic stack 
$\mathcal{M}$ which includes $\mathcal{M}^{o}$ 
as an open substack together with a 
Gorenstein family of polarized K-trivial varieties with 
only canonical singularities $(\mathcal{U},\mathcal{L})\to 
\mathcal{M}$ such that the following holds. 
\begin{enumerate}
\item different $k$-valued points $p_{1}$ and $p_{2}$ of 
$\mathcal{M}$ have non-isomorphic polarized fiber 
$\pi^{-1}(p_{i}) (i=1,2)$. 
\item $\mathcal{M}$ is maximum with respect to 
the inclusion relation among those which satisfies above 

\item the coarse moduli space $M$ of $\mathcal{M}$ 
is quasi-projective. If $k=\C$, its analytification 
coincides with the completion of $M^{o}$ with respect to its 
generalized Weil-Petersson K\"ahler orbi-metrics. 
\end{enumerate}
We call this $\mathcal{M}$ the K-moduli (partial compactification) 
of polarized K-trivial varieties. 
\end{Thm}

We emphasize that this $\mathcal{M}$ is not proper in general so 
is still a ``partial compactification'' of $\mathcal{M}^{o}$. 
Nevertheless, since the parametrized polarized varieties are 
all K-stable (not only K-semistable) 
by \cite{Od}, and also characterized by K-polystability, 
we would call it the K-moduli of 
K-trivial varieties. \footnote{The definition is different from 
the ``over-ambitious'' version long ago 
\cite{Od.oldsurvey, Od.Fano}, 
although many later works show the version works for 
anticanonically polarized $\Q$-Fano varieties case.} 
One reason of the name is that in the K-trivial case, 
the K-polystability and K-stability are equivalent 
(cf., \cite{Od0, Od}). 

\begin{Ex}
In the case if $\mathcal{M}^{o}$ is a moduli of 
polarized smooth irreducible symplectic manifolds, 
the above Theorem \ref{Kmod} is proved with a more 
refined statement by \cite[\S 8.3 Theorem 8.3, Corollary 8.4]{OO} and \cite[Theorem 4.8]{Sch}, 
which shows further that $M$ is nothing but the 
locally Hermitian symmetric space whose 
Weil-Petersson metric is the Bergman metric. 
Note that this refinement crucially depends on the 
Torelli type theorem due to Verbitsky \cite{Ver, Ver.er}. 
\end{Ex}

\begin{proof}[proof of Theorem \ref{Kmod}]
From \cite[Theorem 1.1]{YZha} or \cite[Theorem 8.23 (and 
\S 8.3, \S 8.6)]{Vie}, 
it only remains to show the 
boundedness of the possibly klt K-trivial projective limit of 
the members of $\mathcal{M}^{o}$. 
The following abstract proof of such boundedness 
is a corollary to a combination of known results 
with recent big input due to Birkar \cite{Birkar20}. 
Here we expand the details of the proof. 

We replace $\mathcal{L}^{o}$ by 
$(\mathcal{L}^{o})^{\otimes m}$ for fixed $m\gg 0$ such that 
it is relatively very ample and some effective relatively 
Cartier and relatively smooth 
divisor $\mathcal{D}$ 
exists so that $\mathcal{D}\sim_{\mathcal{M}^{o}} 
(\mathcal{L}^{o})^{\otimes m}$. 
For the sake of simplicity, we can and do assume 
$m=1$. 

For the remained boundedness problem, we take any 
$k$-valued point $s\in \mathcal{M}(k)$ and consider 
corresponding 
polarized variety $\pi^{-1}(s)$. 
We can take a smooth curve $(C,s)\subset \mathcal{M}$ 
passing through $s$ and $\mathcal{M}^{o}$, 
take $\overline{\mathcal{D}|_{(C\setminus s)}}$ 
and its restriction 
$(\overline{\mathcal{D}|_{(C\setminus s)}})|_{s}$ 
as a {\it Weil} divisor. 

An important caution here is that it is a priori a hard problem 
if one can take $m$ uniformly so that $\mathcal{L}|_{\pi^{-1}(s)}$ 
is line bundle (not only $\Q$-line bundle) for {\it all} s, 
or equivalently 
$(\overline{\mathcal{D}|_{(C\setminus s)}})|_{s}$ is 
Cartier for all $s$. Note that whether the latter holds or not 
does not 
depend on the choice of $\mathcal{D}$ due to the equivalence. 

Note that 
by the same arguments as \cite[(proof of) Lemma 2.4]{OSS}, 
it follows that $N$ is $\Q$-Cartier and ample although not necessarily Cartier. 
Now we apply the big result 
\cite[Corollary 1.6]{Birkar20} to the setup $X=\pi^{-1}(s)$, 
$B=0$, $N:=(\overline{\mathcal{D}|_{(C\setminus s)}})|_{s}$ 
in the notation of {\it loc.cit}. From the boundedness assertion, 
one can in particular assume that there is a uniform 
$l\in \mathbb{Z}_{>0}$ such that $lN$ is Cartier, which 
does not depend on $s$. 

Therefore, for instance, if we apply 
the effective basepoint freeness 
\cite[Theorem 1.1]{Kol.bpf} to 
$(\pi^{-1}(s),l(\overline{\mathcal{D}|_{(C\setminus s)}})|_{s})$ 
and then the very ampleness lemma 
\cite[Lemma 7.1]{Fjn.bpf} together with the uniform 
existence of the Castelnuovo-Mumford regularity, 
we obtain the finite typeness of $\mathcal{M}$. 

Further diffential geometric fact that 
the hyperK\"ahler metrics parametrized 
by $M$ is Gromov-Hausdorff continuous with respect to the 
Gromov-Hausdorff topology 
and the fact that $M$ 
the completion of $M^{o}$ with respect to the Weil-Petersson 
metric 
also follows from \cite{YZha} which crucially depends on 
\cite{DS} (see also \cite{Tos.WP} and 
related algebro-geometric issue \cite[\S4]{Od.Fano}). 
\end{proof}

Next we show that weak K-moduli always contains the above 
K-moduli, 
as our terminology may suggest. 

\begin{Thm}\label{include.thm}
For a fixed $\mathcal{M}^{o}$ and the family on it, 
any weak K-moduli $\overline{\mathcal{M}}$ contains 
the K-moduli $\mathcal{M}$ as an open algebraic substack 
with the compatible family on it. 
\end{Thm}

\begin{Rem}
The following proof also shows that 
existence of weak K-moduli compactification 
implies the existence of finite type K-moduli (Theorem \ref{Kmod}) 
without using Birkar's result \cite[Corollary 1.6]{Birkar20} 
as above proof of Theorem \ref{Kmod}. 
\end{Rem}

\begin{proof}
By \cite[Theorem 1.1]{YZha} or  \cite[Theorem 8.23 (and 
\S 8.3, \S 8.6)]{Vie}, 
there is an increasing (and exhausting) sequence of 
moduli substacks of K-moduli stack (which may be a priori 
locally finite type) 
$\mathcal{M}_{1}\subset 
\mathcal{M}_{2}\subset \cdots$. 

We prove that for any $l\in \Z_{>0}$, 
$\mathcal{M}_{l}$ has open immersion $\iota_{l}$ to 
$\overline{\mathcal{M}}$ preserving the ($\Q$-)polarized family. 
Then by the Noetherian argument, the assertion follows. 
We prove the existence of such $\iota_{l}$ by contradiction. 
Assuming the contrary, there should be 
a pointed curve $C\ni p$ together with a morphism 
$\varphi^{o}\colon (C\setminus \{p\})\to 
\mathcal{M}^{o}$ such that 
it extends to both 
$\varphi_{1}\colon C\to \mathcal{M}_{l}$ 
and also $\varphi_{2}\colon C\to \overline{\mathcal{M}}$ 
but not extends to $C\to \mathcal{M}$. 
If we take the 
polarized flat projective 
family $(\X^{o},\mathcal{L}^{o})\to (C\setminus p)$ 
which corresponds to $\varphi^{o}$ 
and its two extensions to $C$ which corresponds to 
$\varphi_{1}$ and $\varphi_{2}$, then the following 
Corollary \ref{CM.min2} (of 
Lemma \ref{CM.min}) 
leads to the contradiction, 
which we reproduce for the sake of convenience.

\begin{Lem}[{cf., \cite[4.2 (i)]{Od.Fano} 
\footnote{but with an obvious typo that 
$K_{\X/C}$ of left hand side meant to be a given polarization 
$\mathcal{L}$ on $\X$ which generically coincides with 
$\mathcal{L}'$.}, \cite[2.14(1)]{Od.Fal}}]\label{CM.min}

Suppose $\pi\colon \mathcal{X}\rightarrow C\ni p$ is a $\mathbb{Q}$-Gorenstein flat projective family of $n$-dimensional semi-log-canonical K-trivial varieties (resp., 
semi-log-canonical K-trivial varieties such that $\pi^{-1}(p)$ is 
log terminal) 
over a projective curve $C$ and a relatively ample line bundle 
$\mathcal{L}$ on $\X$. Take 
any other non-isomorphic 
flat projective family  $\mathcal{X}'\rightarrow C$, 
together with a relatively ample line bundle 
$\mathcal{L}'$, which is isomorphic to 
$(\mathcal{X},\mathcal{L}) 
\rightarrow C$ away from the fiber $\pi^{-1}(p)$. 
Then, we have 
$$
{\rm deg}(\lambda_{\rm CM}(\mathcal{X},\mathcal{L}))\leq (resp.\ <)
{\rm deg}(\lambda_{\rm CM}(\mathcal{X}',\mathcal{L}')). 
$$
Here, $\lambda_{\rm CM}$ stands for the CM line bundles 
(\cite{FS90, PT, FR}) which values at 
${\rm Pic}(C)$ in above situations. 
\end{Lem}

\begin{Cor}[same references and \cite{Bou}]\label{CM.min2}
Consider a punctured ($\Q$-Gorenstein)
family of polarized log terminal K-trivial projective 
varieties 
$(\X^{o},\mathcal{L}^{o})\to (C\setminus p)$, 
where $C$ is a smooth curve and $p$ a closed point, 
 and suppose the existence of its completion 
$(\X,\mathcal{L})\to C$ 
with a log terminal K-trivial fiber $\pi^{-1}(p)$. 
Then, 
there is no other completion 
$(\X,\mathcal{L})\to C$ 
with semi-log-canonical K-trivial $\pi^{-1}(p)$. 
\end{Cor}

We conclude the proof of Theorem \ref{include.thm}. 
\end{proof}

%%%%%%%%%%%%%%%%%%%

\section{Log minimality of weak K-moduli}\label{wKmod.sec}

\subsection{Statements}

We now proceed to the log minimality discussion. 
Our setup and the list of assumptions is as follows which we believe 
to be ubiquitous. Indeed, examples include those recently 
constructed in \cite{AET, AE} (also cf., \cite{ABE}). 
In this section and next section,  
we assume $k=\C$. 

\subsubsection{Notation and Assumptions}\label{Notation}

In this section, we work on the {\it normalization of}  
weak K-moduli and 
put the following natural assumption and notation. 
Although we take normalization, we still denote it as 
$\mathcal{M}^{o}$ for simplicity. 
The following conditions are designed to fit to 
various known explicit examples of the compactifications, as we explain later. 

\begin{itemize}

\item 
$G$ is a finite group, which could be a priori trivial. 

\item $S$ is a klt normal projective variety, 
$S^{o}$ is its Zariski open subset, 
$\Delta=S\setminus S^{o}$ is purely codimension $1$ such that 
$G$ acts on $S$ preserving $\Delta$ such that 
$(S,\Delta)$ is log canonical. For instance, any 
smooth $S$ with normal crossing divisor $\Delta$ is allowed. 

\item K-trivial varieties are parametrized by: 
$\pi\colon \mathcal{X}\to S$ a $G$-equivariant flat 
proper morphism from an algebraic space $\X$, 
with the relative dimension $n$, 
such that 
\begin{enumerate}
\item For any $s\in S^{o}$, $\pi^{-1}(s)$ is K-trivial 
projective variety with only 
canonical singularities, 
\item there is a $\pi|_{S^{o}}$-relatively ample 
line bundle $\mathcal{L}^{o}$ on $\pi^{-1}(S^{o})$, 
\item for any closed point $s\in \Delta=S\setminus S^{o}$, the 
$\pi$-fiber $\pi^{-1}(s)$ is again a K-trivial variety 
with only semi-log-canonical singularities 
\end{enumerate}
\begin{Rem}
Some technical remarks are in order here. 
Note that we are not assuming $\X$ itself is a variety 
while only fiberwise algebraicity is assumed. 
Compare with the classical Kulikov model situation 
(cf., \cite{Kul},\cite{PP}, 
Theorem \ref{KPP.geom} \eqref{nn.alg} in the appendix). 
From our singularities assumption, note that 
any $\pi^{-1}(s)$ for 
$s\in S^{o}$ (resp., $s\in \Delta$) 
is K-stable (resp., K-semistable) 
for any polarization by \cite{Od}. 

Also $\pi$ is automatically a ($\Q$-)Gorenstein family 
and locally stable in the sense of Koll\'ar \cite{Kol}. 
\end{Rem}

\item 
On each log canonical center $T$ of $(S,\Delta)$, whose support is  automatically inside 
inside $\Delta$, write $\pi^{-1}(T)$ as $\X_{T}$. Then, 
%either one of 
the following hold: 
%\begin{itemize}
%\item 
%$\X_{T}\to T$ is analytically locally trivial 
%(e.g., if the total space is simple normal crossing whose 
%strata are all smooth over $T$), 
%\item or more generally, if 
there is a birational proper modification 
$\X_{T}'\to \X_{T}$ such that 
the composite $\X_{T}'\to T$ is an analytically 
locally trivial fibration with snc K-trivial fibers. 
We assume either the relative dimension $n$ is $2$ or 
that for general points $s$ of $S$, $\X_{t}$ is 
irreducible holomorphic symplectic manifold and the 
fibers of $\X_{T}'$ occur as good degenerations 
in the sense of \cite[4.2]{Nag.mon}. 

(For $n=2$, imagining with Brieskorn simultaneous resolution of 
deforming ADE singularities may help understanding. 
See also our Appendix.) 
%\end{itemize}

\item We assume 
$\overline{\mathcal{M}}:=[S/G]$ is the normalization of some 
moduli stack $[S'/G]$ i.e., 
closed points $s, s' \in S'$ parametrize the same projective surface 
if and only if $Gs=Gs'$. The naturality of taking 
normalization can be seen e.g. in \cite{ABE}. 

\item By Keel-Mori theorem \cite{KeelMori}, $\overline{\mathcal{M}}$ has a coarse moduli algebraic space 
$\overline{\mathcal{M}}\to \overline{M}$ which we assume to be projective. 

\item Denote the irreducible decomposition of the 
branch divisor of $S\to M$ as $\cup_{i} D_{i}$ 
such that the branch degree along $D_{i}$ is $m_{i}$. 
We set $D:=\sum_{i}\frac{m_{i}-1}{m_{i}}D_{i}$. 

\end{itemize}

\begin{Thm}[Log minimality]\label{log.min}
Under the above assumptions, $\overline{M}$ 
is log minimal with the natural 
boundary $\Q$-divisor i.e., $K_{\overline{M}}+D+\Delta$ is nef so that 
$(\overline{M},D+\Delta)$ is a (log canonical) log minimal model. 
Equivalently, $K_{S}+\Delta$ is nef. 
\end{Thm}

\begin{Rem}
A related very recent result for the log smooth case 
under local Torelli type assumption, 
whose proof uses curvature consideration 
appears in \cite[Theorem 3.1($=$1.17)]{GGR1}.
\end{Rem}

\begin{Rem}
One can also regard this theorem as a variant of another recent 
result 
\cite[Theorem 1.2]{AE} which claims certain 
geometric compactifications under similar conditions are 
semi-toric in the sense of Looijenga \cite{Looi2}, which are 
log minimal as easily follows from \cite[7.18]{AE}. 
On the other hand, the construction of $\X_{T}$ 
in \cite{AET, ABE, AE} starts with $\X_{T}'$ and then 
use contraction to obtain $\X_{T}$ as their key idea. 
As it follows from the proof, Theorem 
\ref{log.min} holds also if we weaken the maximal variation-ness of 
$\X'_{T}$ over $T$ {\it for each (fixed)} $T$. 
\end{Rem}

The following provides another partial reconstruction of 
Satake-Baily-Borel compactification or its analogue a priori. 

\begin{Cor}[Alternative construction of Baily-Borel compactification]
We work under the same setup as above \S \ref{Notation}, 
Theorem \ref{log.min}. 
If $(S,\Delta)$ is dlt and $\dim(S)\le 4$, 
the normalized weak K-moduli $\overline{M}=S/G$ has a birational 
morphism to its log canonical 
model $\overline{M}^{lc}$ which exists. If 
$S^{o}$ is a locally Hermitian symmetric space, 
the lc model $\overline{M}^{lc}$ coincides with the 
Satake-Baily-Borel compactification. 
\end{Cor}

\begin{proof}
The existence of log canonical model 
$\overline{M}^{lc}$ follows from \cite{Fjn4} 
and the proof that $\overline{M}\dashrightarrow 
\overline{M}^{lc}$ is a morphism follows from 
Theorem \ref{log.min}. 
\end{proof}
This gives a rather partial confirmation of 
\cite[Conjecture B.1 (ii)]{galaxy}, but it also seems to be the first 
arguments, which do not logically use the theory of Shimura varieties. 

\subsection{Admissible variation of mixed Hodge structures}

Our approach to Theorem \ref{log.min} 
is partially algebraic and partially analytic. 
The algebraic part is connected to a logarithmic refinement of the 
cone theorem. In particular, we use the log minimality criterion 
\cite{Sva, Fjn.hyp} via 
Mori hyperbolicity which was introduced in 
\cite{LZ}. The analytic part involves variation of 
{\it mixed} Hodge structures. 
Indeed, 
by using {\it loc.cit}, we can and do reduce Theorem \ref{log.min} to 
the following claim; 
Theorem \ref{nosncfamily} on the isotriviality of certain 
singular families and then prove it by crucially 
using variation of mixed Hodge structures. 
Its relatively smooth case is known by \cite[Theorem 1.4 (iii)]{VZ2} 
as its Corollary, but the point below is that 
we generalize it to allow even 
non-normal fibers so that the same results as smooth case 
(\cite{VZ1, VZ2})  
do not literally hold. 
Also note that taking normalization does not 
reduce to normal case, except for the case \eqref{curve}, 
since the variation of 
glueing data is usually important here. 

\begin{Thm}\label{nosncfamily}
Consider any projective flat 
family $f\colon \mathcal{X}\to \mathbb{A}^{1}$ of relative dimension $n$, such that 
$\X$ is simple normal crossing variety whose strata maps smoothly over $\mathbb{A}^1$, 
\footnote{in particular, it is basic slc-trivial fibration 
in the sense of Fujino} and 
(relative) dualizing sheaf is trivial i.e., 
$\omega_{\mathcal{X}/\mathbb{A}^{1}}\sim_{/\mathbb{A}^{1}} 
\mathcal{O}_{\X}$. 
Then it follows that 
for the irreducible decomposition $\X=\cup_{i} \X_{i}$, 
at least one strata $S_{I}:=\cap_{i\in I}\X_{i}$ with 
its natural boundary divisor 
$D(S_{I}):=
(\cup_{j\notin I}\X_{j})\cap S_{I}$ 
is isotrivial i.e., 
the isomorphic class of the 
pair 
$$(S_{I}\cap \pi^{-1}(t), D(S_{I})\cap \pi^{-1}(t))$$ 
 does not depend on $t$. 

Furthermore, 
assume that at least one of the following holds. 

\begin{enumerate}
\item \label{curve} 
$n=1$ (in this case, 
$\omega_{\mathcal{X}/\mathbb{A}^{1}}\nsim_{/\mathbb{A}^{1}} 
\mathcal{O}_{\X}$ is also allowed) or 
\item  \label{k3}
$n=2$ (not assuming d-semistability of the fibers i.e., they may not be smoothable), or  
\item \label{hk}
$n$ is even and 
the fibers $X_{t} (t\in \A^{1})$ 
occur as good degenerations of holomorphic symplectic 
varieties in the sense of \cite[Definition 4.2]{Nag.mon} and its mixed Hodge structures on $H^{2}(\X_{t},\Z)$ \cite{Del.III}
are \underline{not} constant i.e., local (mixed) Torelli theorem holds 
for $\X/\A^{1}$. 
\end{enumerate}

Then $f$ is isotrivial i.e., 
the closed fibers $\X_{t}$ of $f$ are isomorphic to each other. 
\end{Thm}
Here are some remarks especially on the latter assertion of 
isotriviality of $f$ in case \eqref{k3}. 
\begin{Rem}[On the projectivity assumption]
We expand the details of this remark at 
our Appendix \S \ref{Appendix}. 
Recall that for 
any projective semistable family $\Y$ over a curve 
whose general fibers are K3 surfaces, 
the result of Kulikov-Pinkham-Persson's degeneration 
\cite{Kul, PP} (reviewed as Theorem \ref{KPP.original}) 
is {\it not} necessarily a projective family. 
However from our Theorem \ref{KPP.geom} and its proof using MMP, 
it follows that the closed fibers $\tilde{\X}_{0}$ are all projective even in such a case i.e., fiberwise projectivity holds. 
This explains the naturality of our projectivity assumption. Indeed, 
various part of Theorem \ref{KPP.geom} naturally extend over higher dimensional base case. 
\end{Rem}

\begin{Rem}
Recall that, for instance if $n=2$ with nonempty $0$-dimensional strata 
i.e., such as family of maximally degenerated 
K3 surfaces or abelian surfaces, 
the period spaces 
of such simple normal crossing K-trivial 
surfaces are known to be often algebraic torus 
(cf., e.g., \cite{Carlson, FS}). 
Such arguments may look implying the above Theorem \ref{nosncfamily}. 
However, for instance, there is still a exponential map from $\C$ to 
algebraic torus which could be a priori the (mixed) period map, 
we have to exclude such possibility. We need 
more generalized version of such claim 
for the application to moduli 
compactifications, which is the point of above theorem and its formulation. 
\end{Rem}

\begin{Rem}
One can slightly weaken the locally triviality assumption of 
Theorem \ref{nosncfamily}. 

For instance, suppose $n=2$ and 
neighborhood of 
normal crossing 
singular locus are topologically locally trivial over $\A^{1}$. 
Further, suppose 
$f$ extends to $f\colon \overline{\X}\to \PP^{1}$ and 
the fibers $f^{-1}(t)$ only normal crossing singularities and 
also finite ADE singularities for at most one $t(\neq \infty)$. 
Now we employ the Brieskorn simultaneous resolution for the 
relatively ADE locus to 
reduce to analytically locally trivial case. 
(Compare with our Appendix.) 
Note that for the simultaneous resolution, 
we do not need base change of $\A^{1}$ in such case, 
since the monodromy around $\X_{t}$ with possibly ADE singularities 
is trivial. The triviality follows from its unipotency because 
$\overline{\X}_{\infty}$ is reduced normal crossing by 
\cite{Landman} and also of finite order. 
\end{Rem}

For the actual proof of Theorem \ref{nosncfamily} in 
the above generality, we analyze admissible 
variation of mixed Hodge structures 
(also known as smooth mixed Hodge modules) associated to 
our singular family. 

\begin{proof}[Proof of Theorem \ref{nosncfamily}]
The proof of the former assertion 
follows from \cite{Deng} 
which we apply to the deforming strata (as a log pair). 
Indeed, 
if we consider the normalization $S_{I,t}^{\nu}$ 
of the fibers $S_{I,t}$ of $S_{I}\to \A^{1}$, denoted as 
$\nu_{I,t}\colon S_{I,t}^{\nu}\to S_{I,t}$, the log pairs 
$(S_{I,t}^{\nu}, \nu_{I,t}^{-1}({\rm Sing}(S_{I,t})
+D(S_{I,t})))$ becomes log Calabi-Yau pairs discussed in \cite{Deng} 
by a simple subadjunction. Hence we can apply 
\cite[Theorem A]{Deng} to show our former assertion of 
Theorem \ref{nosncfamily}. 
\vspace{3mm}

For the latter assertion of the isotriviality of $f$, 
we first prove the cases \eqref{k3} and \eqref{hk}, 
and later the case \eqref{curve} for better comparison. 

Consider 
$\mathcal{V}:=R^{2}f_{*}f^{-1}\mathcal{O}_{\A^{1}}$ 
which 
admits 
an admissible variation of mixed hodge structure 
$(\mathcal{V},F^{\cdot}, 
W_{\cdot})$ due to 
\cite{SZ}, \cite{ElZein}, \cite{Kashiwara}, \cite[4.15]{FF}. 
For simplicity, by taking base change with respect to 
$m$-th power map $\A^{1}\to \A^{1}$ sending $t$ to $t^{m}$ for 
certain $m\in \Z_{>0}$, 
we can and do assume that the monodromy of $W$ is unipotent. 
The above admissibility is in the sense of \cite{SZ, Kashiwara}, 
which in particular asserts that 
for the Deligne canonical extension %(\cite{Del70}) 
$\overline{\mathcal{V}}$ 
of $\mathcal{V}$, we can take 
subbundles integral $\overline{W_{\cdot}}$ 
as the canonical extensions of $W_{\cdot}$ and 
and locally free subbundles of $\overline{\mathcal{V}}$ 
as $\overline{\mathcal{F}_{\cdot}}$. 
We consider the following Higgs field 
\begin{equation}\label{Higgs}
\overline{F}^{1}/\overline{F}^{2}\xrightarrow{\theta_{1}}
\overline{F}^{0}/\overline{F}^{1}\otimes \Omega_{\PP^{1}}^{1}({\rm log}([\infty])),
\end{equation}
which comes from the Gauss-Manin connection, 
or in other words, cup product with the Kodaira-Spencer 
section. To make it explicit, 
$$F^{1}/F^{2}\simeq R^{1}f_{*}\tilde{\Omega}_{\X/\A^{1}}^{1},$$
$$F^{0}/F^{1}\simeq R^{2}f_{*}\mathcal{O}_{\X},$$
where $\tilde{\Omega}_{\X/\A^{1}}^{1}:=
\Omega_{\X/\A^{1}}^{1}/\tau^{1}_{\X},$, 
$\tau^{1}_{\X}:={\rm tors}(\Omega_{\X/\A^{1}}^{1})$ (torsion subsheaf, supported on the singular locus of $\X$). 
After usual notation, we put 
$\mathcal{T}^{o}_{\X/\A^{1}}:=(\Omega_{\X/\A^{1}}^{1})^{\vee}
=(\Omega_{\X/\A^{1}}^{1} /\tau^{1})^{\vee}$. Then 
the Kodaira-Spencer class of $f$ gives rise to a section $s$ of 
$R^{1}f_{*}\mathcal{T}^{o}_{\X/\A^{1}}$ and $\theta_{1}$ is the 
cup product with $s$. 

Suppose that $f$ is not isotrivial so that $s$ is nontrivial 
by the locally trivial version of deformation theory. 
In our case, the locally free sheaf 
$F^{1}/F^{2}$ is Grothendieck dual to 
$R^{1}f_{*}\mathcal{T}^{o}_{\X/\A^{1}}$ hence the existence of 
nonzero $s$ 
asserts that $F^{1}/F^{2}$ is also not vanishing. 
On the other hand, as \cite{FF} shows for instance, 
$F^{0}/F^{1}\simeq (f_{*}\omega_{\X/\A^{1}})^{\vee}$ 
which is an invertible sheaf. 

Note that in \eqref{k3} case, 
$n=2$ assumption implies a natural inclusion 
$F^2\subset f_* \omega_{\X^{\nu}/\mathbb{A}^1}$ 
from the description via standard Hodge type spectral sequence (cf., e.g., \cite[1.5]{Fri.dss}) 
where $\X^{\nu}$ stands for the normalization of $\X$ which is smooth over $\A^1$. We can and do 
assume that $\X$ is not smooth since otherwise 
\cite[Theorem 1.4 (iii)]{VZ2} applies. Then, 
since each component of closed fibers of $\X^{\nu}$ should have nontrivial effective anticanonical divisor, 
we have $f_* \omega_{\X^{\nu}/\mathbb{A}^1}=0$ hence $\overline{F}^2=0$. 

Similarly, in \eqref{hk} case, we have 
$F^2\subset f_* \Omega_{\X^{\nu}/\A^{1}}^{2}$ 
whose right hand side vanishes by the assumption. 
Also, if $\X$ is non-isotrivial which we assume here, 
the local Torelli type assumption for \eqref{hk} implies that the 
Higgs field $F^{1}\to F^{0}/F^{1}$, induced by the 
Gauss-Manin connection, is automatically nontrivial. 

In both cases, either \eqref{k3} and \eqref{hk}, the seminegativity results \cite[1.3]{FF}
following the method of the curvature formula of Hodge metrics due to Griffiths \cite{Gri, Zucker}, 
show 
$\overline{F}^{1}=\overline{F}^{1}/\overline{F}^{2}$ and 
$\overline{F}^{0}/\overline{F}^{1}$ are both seminegative. 
This contradicts with \eqref{Higgs} since $\Omega^1_{\mathbb{P}^1}({\rm log}([\infty]))$ has negative degree ($-1$). 
This completes the proof of Theorem \ref{nosncfamily} 
cases \eqref{k3}, \eqref{hk}. 

Now we move on to the (relatively) $1$-dimensional case \eqref{curve}. 
In this case, we do not need to assume $\X/\A^{1}$ is K-trivial. 
We consider the $\mathcal{V}:=R^{1}f_{*}f^{-1}\mathcal{O}_{\A^{1}}$ 
which again 
admits 
an admissible variation of mixed hodge structure 
$(\mathcal{V},F^{\cdot}, 
W_{\cdot})$. Then similarly as above, 
$F^{1}=f_{*}(\Omega_{\X/\A^{1}}/\tau_{\X/\A^{1}})$, 
where $\tau_{\X/\A^{1}}$ denotes the torsion of 
$\Omega_{\X/\A^{1}}$. $F^{1}$ is also canonically isomorphic to 
$(f\circ \nu)_{*}\omega_{\X^{\nu}/\A^{1}}$, where 
$\nu\colon \X^{\nu}\to \X$ denotes the 
normalization of $\X$. Also, $F^{0}/F^{1}
\simeq R^{1}f_{*}\mathcal{O}_{\X}\simeq 
(f_{*}\omega_{\X/\A^{1}})^{\vee}$. 

Suppose the contrary of the assertion i.e., 
non-isotriviality of $\X$ over $\A^{1}$. Note that since 
the base $\C$ is simply connected, the combinatorial structure 
i.e., dual graph of the fibers are canonically identified to 
each other. 
Then the 
fibers of $\X^{\nu}$ with natural marking on 
$\nu^{-1}{\rm Sing}(\X/\A^{1})$ where ${\rm Sing}(\X/\A^{1})$ 
denotes the (relatively) non-smooth locus, should contain 
non-isotrivial family of marked smooth curves. 
For every rational components of the fibers of $\X^{\nu}$ over $
\A^{1}$, since any four marked points 
gives rise to a nonvanishing cross ratio, we conclude 
that $\PP^{1}$ in the normalization together with the marked points  
are (iso)trivial. 
On the other hand, for any non-rational components $Z_{t}$ of the 
fibers of $\X^{\nu}$, 
the same arguments to the above proof of cases \eqref{k3}, 
\eqref{hk} show $Z_{t}$ itself is isotrivial 
since $F^{1}$ deforms nontrivially 
in $R^{1}f_{*}f^{-1}\mathcal{O}_{\A^{1}}$ otherwise, 
to lead to the contradiction. (Or apply \cite{VZ1,VZ2}.) 
Furthermore, since the base is rational it also follows that 
the natural marked points in the 
components $Z_{t}$ are again 
(iso)trivial. Therefore, combining these, 
we again conclude that $\X$ is isotrivial and 
hence the proof of the latter assertion of 
Theorem \ref{nosncfamily} for \eqref{curve} case. 
\end{proof}

\begin{Rem}
The above proof applies rather partially for $n>2$ under ``artificial" assumptions (e.g., 
$n=3$ and all components $V_i$ satisfying $h^{2,1}(V_i)=0$ such as $\mathbb{P}^3$) 
but we omit it here as we have no particular application. 
\end{Rem}

Now we are ready to combine our results to confirm 
the desired log minimality. 
\begin{proof}[proof of Theorem \ref{log.min}]
We use \cite[1.1, 6.9]{Sva} or \cite[1.4, 9.1]{Fjn.hyp} 
to reduce the proof to the claim that 
each log canonical center of $(S,\Delta)$ does not allow any 
non-constant morphism 
from the affine line $\A^{1}$. On the other hand, 
our theorem \ref{nosncfamily} using admissible 
variation of mixed Hodge structures implies it 
by our fourth assumption 
in the Notation \ref{Notation}. This completes the proof. 
\end{proof}

\begin{Rem}
As an analogue of Theorem \ref{log.min}, 
Theorem \ref{nosncfamily} \eqref{curve} combined with 
\cite{Sva, Fjn.hyp} it follows that 
$(\overline{\mathcal{M}_{g}}^{\rm DM},
\overline{\mathcal{M}_{g}}^{\rm DM}\setminus \mathcal{M}_{g})$ 
is at least log minimal. 
Here, $\mathcal{M}_{g}$ refers to the moduli stack of 
smooth projective curve of genus $g\ge 2$ and 
$\overline{\mathcal{M}_{g}}^{\rm DM}$ refers to its 
Deligne-Mumford compactification. What we meant by log minimality at the stacky level above can be also rephrased as follows: 
if we take a coarse moduli scheme 
$\mathcal{M}_{g}\to 
M_{g}$ branches at $D_{i}$ with degree $b_{i}$, 
the above log minimality is equivalent to that of 
$(\overline{M_{g}}^{\rm DM},\sum_{i}\frac{b_{i}-1}{b_{i}}\overline{D_{i}}+(\overline{M_{g}}^{\rm DM}\setminus M_{g}))$. 
This 
partially recover \cite[Theorem 1.3]{CH} 
by a fairly different method. 
\end{Rem}

%%%%%%%%%%%%%%%%%%%%%%%%%%%%%%%%%%%%
%%%%%%%%%%%%%%%%%%%%%%%%%%%%%%%%%%%%
%%%%%%%%%%%%%%%%%%%%%%%%%%%%%%%%%%%%
%%%%%%%%%%%%%%%%%%%%%%%%%%%%%%%%%%%%
%%%%%%%%%%%%%%%%%%%%%%%%%%%%%%%%%%%%
%%%%%%%%%%%%%%%%%%%%%%%%%%%%%%%%%%%%
%%%%%%%%%%%%%%%%%%%%%%%%%%%%%%%%%%%%

\appendix
\section{Algebraic construction of Kulikov models}\label{Appendix}

\subsection{Review of history and original statements}
The following classical theorem is our topic of this appendix, 
which is indirectly related to the problem of weak K-moduli 
(see e.g., \cite{AE}). 

\begin{Thm}[{\cite[Theorem I]{Kul},\cite[Theorem]{PP}}]\label{KPP.original}
$\pi\colon \mathcal{X}\to \Delta=\{t\in \mathbb{C}
\mid |t|<1\}$ be a proper flat family of complex analytic 
surfaces such that 
\begin{enumerate}
\item general fibers $\mathcal{X}_{t}$ are smooth 
with trivial canonical line bundles 
\item central fiber $\X_{0}$ is (reduced) simple normal crossing with 
all irreducible components algebraic 
\end{enumerate}
Then there is another model $\tilde{\X}$ which coincides away from $\X_{0}$ 
which satisfies both two above conditions and further that 
$K_{\tilde{\X}/\Delta}\sim 0$. 
\end{Thm}
Recall their construction was highly topological with respect to 
complex analytic topology. 
In \cite[Theorem I]{Kul}, an additional assumption that the 
family is projective is put but 
the structures of the degenerations are also determined. 
%(see also \cite[I, Appendix]{Fri}). 
Such determination of degeneration types is also similarly done in 
arithmetic setting (cf., e.g., \cite[3.4]{Nakk}). 

The purpose here is to recover and generalize  
the above model construction, 
by means of {\it the minimal model program}, 
which originally emerged shortly after 
\cite{Kul, PP} 
and much developped during these several decades. 

Hence the idea is fairly natural and simple. 
The author once believed such arguments had been naturally 
expected as folklore 
or perhaps even known to 
some experts of birational geometry as he indeed felt during 
several conversations with experts. 
Nevertheless, the only literature the author 
has been able to find so far is the nice  series of works by 
Yuya Matsumoto and Christian Liedke in a more 
arithmetic context. See the proofs of 
\cite[3.1]{Mat}, \cite[3.1]{LM}, 
which we review at \ref{mix} later. Their papers 
\cite{Mat, LM} indeed went beyond; 
establishing good reduction criteria for arithmetic K3 surfaces 
as in abelian varieties case, rather than just 
(re)constructing Kulikov model in general. 
This notes only mean to fill the apparent 
lack of (more) complete reference 
on this matter, as a slight refined version
 disallowing a base change. 
See Theorem \ref{KPP.geom} below and their proofs 
for the details. 
\vspace{5mm}

\subsection{Proof of algebraic case}
We first (re)construct the Kulikov model via the minimal model program,  
with the base field an algebraically closed field $k$ of 
{\it any} characteristics. 

\begin{Thm}[Kulikov model - algebraic reconstruction]\label{KPP.geom}
Suppose the base field $k$ is algebraically closed. Take 
any strictly semistable 
 proper scheme of dimension $3$ 
over a smooth $k$-curve $C$, denoted as 
$\pi \colon \X\to C$, 
whose general fibers have trivial canonical divisors. 
Then, there is a birational map from an 
algebraic space 
$\tilde{\X}\dashrightarrow \X$ such that 
\begin{enumerate}
\item 
strictly semistable 
\item 
$K_{\tilde{\X}/C}\sim_{C} 0$. 
\item \label{nn.alg}
all closed fibers of $\tilde{\X}\to C$ are projective algebraic. 
\end{enumerate}
\end{Thm}
\noindent 
Note that the above statements do 
{\it not} involve base change of $C$. 
Also note that the point \eqref{nn.alg}
 gives slight improvements of the component-wise algebraicity 
 in original \ref{KPP.original}. 
 For extensions over more general base with similar arguments, 
 we refer to \S \ref{generalbase}. 
 
\begin{proof}[proof of Theorem \ref{KPP.geom}]

First we replace $\X$ by its blow up which is projective over $C$ and its further projecive log resolution 
so that one can assume $\X$ is projective over $C$. Then we replace $\X$ by its minimal model by either 
\cite{Fjn.ss} (for characteristic $0$) or \cite{HX, HW} (for positive characteristics). Our proof depends on analysis of the 
possible singularities on a relative minimal model of $\X$ over $C$, 
which we denote $\pi_{\min}\colon \X_{\rm min}\to C$. 
Our first step is as follows, which is not discussed in 
\cite{kwmt, Mat, LM} because \cite{kwmt} assumes  
a weaker version of $\Q$-factoriality which excludes most subtle singularities. 

In this notes, dlt minimal model $\pi_{\min}\colon \mathcal{X}_{\rm min}\to C$  only means that $(\X_{\min},\X_{\min,0})$ is dlt, 
$K_{\X_{\min}}(+\X_{\min,0})\sim_{C,\Q} 0$, 
without $\Q$-factoriality for generalization 
(e.g. to allow all nodes). The next proposition is a first step, 
which also aims at understanding the local strutures of all 
non-necessarily $\Q$-factorial dlt models. 

\begin{Prop}[Small resolution as algebraic space]
\label{Step1}
For any dlt minimal model 
$\pi_{\min}\colon \mathcal{X}_{\rm min}\to C$, 
%whose 
%generic fiber is smooth and has trivial canonical line bundle, 
%we have the following modification: 
there is a small resolution from relatively proper (over $C$) 
algebraic space 
$\mathcal{Y}\to \mathcal{X}_{\rm min}$ such that any closed 
point $p\in \Y$ satisfies one of the followings. 
Here, the composite $\Y\to C$ is denoted as $\pi_{\Y}$. 
\begin{enumerate}
\item \label{snc} $\Y$ is regular at $p$ and the fiber 
$\pi_{\Y}^{-1}(p)$ is normal crossing around $p$ or 
\item \label{rdp} $\pi_{\Y}^{-1}(p)\ni p$ is a rational double point. 
\end{enumerate}
\end{Prop}

\begin{proof}[proof of Proposition \ref{Step1}]

We first analyze the local structure of $\mathcal{X}_{\rm min}$. 
While we make some arguments partially self-contained, 
they are essentially in \cite{Reid}, 
\cite{Mori} as 
their special cases. 
Since singular fibers lie over only finite closed points in 
$C$, we can focus on one of them which we denote as $0\in C$. 
We denote its fiber as $\X_{0}$. 

If one takes the germ of any (but automatically terminal) 
singularity $p\in 
\mathcal{X}$, $\mathcal{X}_{\min,0}$ is Gorenstein by the 
upper semicontinuity of ${\rm dim}(H^{0}
(\omega_{\mathcal{X}_{\min}/C}|_{\X_{\min,t}}))$ 
(cf., \cite{Fjn.ss}). Therefore, from the adjunction of 
$\mathcal{X}_{\min}$ 
to $\mathcal{X}_{\min,0}$, 
it also follows that $\mathcal{X_{\min}}\ni p$ is also Gorenstein,
not only $\mathbb{Q}$-Gorenstein. Also note that 
$\mathcal{X}_{\min,0}$ is semi-dlt in the sense of 
\cite[5.19]{Kol13} which is stronger than that of 
\cite[1.1, 1.2]{FjnM}. 

Its general hyperplane passing through $p$ obtains DuVal singularity 
at $p$, by tha classics of M.Reid, S.Mori in 
characteristics $0$ (cf., e.g., \cite{Reid}, \cite{Mori}, 
\cite[\S 5.3]{KM}) which also holds in positive characteristics 
due to recent \cite{ST}. 
Hence $p\in \X_{\min}$ itself is locally a hypersurface singularity, 
so that we naturally wish to analyze 
the local equation in the completion of 
$\mathcal{O}_{x,\X_{\min}}$. 

We take a uniformizer of $0\in C$ as $u$, 
which we can regard as an element of $\mathcal{O}_{x,\X_{\min}}$ 
or its completion $\widehat{\mathcal{O}_{x,\X_{\min}}}$
via $\pi_{\min}$. 

In our situation, we know that $p\in \X_{\min}$ 
is Gorenstein. If 
$p\in (u=0)\cap \X$ passed through by three 
smooth irreducible components, 
then from sdlt condition it is locally 
strictly simple normal crossing thus in the situation of 
\eqref{snc}. Also, if $p\in (u=0)\cap \X$ is normal 
then it is in the situation of \eqref{rdp}. 
If $p\in \pi_{\Y}^{-1}(0)$ is an irreducible germ, 
it is klt and Gorenstein hence rational double point i.e., 
in the situation of \eqref{rdp}. 

Suppose otherwise - there would be exactly two irreducible components 
$V_{i}\ni p (i=1,2)$ 
passing through $p\in (u=0)\cap \X$, with the double locus 
$D=V_{1}\cap V_{2}$. Then, the germ $(V_{i},D)$ is plt, 
hence its formal germ is same as the cyclic quotient singularity 
$\frac{1}{n}(1,r)$ in $\mathbb{A}^{2}_{x,y}$ 
where $D$ is the vanishing locus of the 
first coordinate $x$ and ${\rm gcd}(r,n)=1$ 
(cf., e.g., \cite[3.31, 3.32]{Kol13}). 
A direct self-contained explanation to confirm it is 
to take the index $1$ covering of $V_{1}$, 
which is a cyclic covering, to make 
the pullback of $V_{1}\cap V_{2}$ Cartier and regular again. 
Thus the cyclic cover itself is smooth and we observe the above 
description. 

As the Gorensteinness of $\X_{\min,0}$ implies, 
${\rm Res}_{(x=0)}(\frac{dx}{x}\wedge dy)|_{y=0}=dy|_{y=0}$ 
needs to be $\mu_{n}$-invariant, hence $n=1$ so that 
$V_{i}\ni p (i=1,2)$ are 
both smooth at $p$. 
This implies that 
\begin{align}\label{modu}
\widehat{\mathcal{O}_{p,\X_{\min}}}/(u)
\simeq k[[x,y,z]]/(xy). 
\end{align}
Since the embedded dimension of $p\in \X_{\min}$ is $4$, 
$u$ gives a regular element of $\widehat{\mathcal{O}_{p,\X_{\min}}}$ 
and any lift of $x,y,z$ denoted by the same letters 
complements the system of parameters of 
$\mathcal{O}_{p,\X_{\min}}$ 
as $x,y,z,u$. Hence, there is a $F\in k[[x,y,z,u]]$ such that 
%(see e.g. \cite[\S28, \S29]{Matsumura})
\begin{align}\label{completion.singularity}
\widehat{\mathcal{O}_{p,\X_{\min}}}
\simeq k[[x,y,z,u]]/F(x,y,z,u). 
\end{align}
From \eqref{modu}, 
\begin{align}
F(x,y,z,u)=xy+uf(x,y,z,u).
\end{align}
%Before going to further analysis, we make a remark 
%which eliminates the possibility in $\Q$-factorial case. 
%\begin{Claim}
%If $\X_{\min}$ is $\Q$-factorial, this type of singularities 
%do not exist. 
%\end{Claim}
%\begin{proof}
%Indeed, if $mV_{1}$ is Cartier divisor for some 
%$m\in \mathbb{Z}_{>0}$, it means 
%$(x,u)^{m}$ is principal 
%in the completion 
%\eqref{completion.singularity}. Thus, there are formal power series 
%$h(x,y,z,u)$, $r_{x}(x,y,z,u)$, $r_{u}(x,y,z,u)$, 
%$q_{x}(x,y,z,u)$, and 
%$q_{u}(x,y,z,u)$ such that 
%$$x^{m}=r_{x}(x,y,z,u)h(x,y,z,u)+(xy+uf(x,y,z,u))q_{x}(x,y,z,u),$$
%$$u^{m}=r_{u}(x,y,z,u)h(x,y,z,u)+(xy+uf(x,y,z,u))q_{u}(x,y,z,u).$$

%\end{proof}
We note that the above type singularity is (even algebraically) non-$\Q$-factorial, 
as it follows when we consider the blow up of the prime divisor which is formally locally 
$(x,u)$ in the above. 
Hence, in $\Q$-factorial assumption, 
we can shorten the following arguments. 
Note in particular that, in the case $\X_{\min}$ is obtained by 
running relative MMP over $C$ from regular semistable model $\Y$, 
it is $\Q$-factorial (see \cite[3.18, 3.37]{KM} whose 
proofs work over any field). 

We continue the analysis of the above type (non-$\Q$-factorial) singularity, 
If $f\equiv aux+buy {\rm (mod.}\texttt{m}_{x,\X_{\min}}^{2})$ 
for some $a, b\in k$, by replacing $x,y$ by 
$x+bu, y+au$ (which we still denote by the same letters, 
avoiding complications) the equation becomes 
\begin{align}\label{eqn}
xy+uxh_{x}+uyh_{y}+ug(z,u),
\end{align}
where $h_{x}, h_{y}, g \in k[[x,y,z,u]]$ 
and $g$ involves the variables only $z, u$. 
By further replacing $x,y$ by $x+uh_{y}, y+uh_{x}$, 
we can and do assume $h_{x}=h_{y}=0$. 
Hence, we obtain a normal form of $F$ as 
\begin{align}
F(x,y,z,u)=xy+ug(z,u),
\end{align}
where $g(z,u)$ does not contain non-zero constant i.e., 
not unit. 

Note that this is essentially proven as a special case in 
\cite[Theorem 12 (also cf., Theorem 3, Corollary 4)]{Mori}. 
%In \cite{Mori} the arguments partially rely on approximation argument 
%\cite[5.11]{Art} \cite{Tou}) which we follow. 
Also note that without Gorenstein property, a priori 
it would possibly have been 
$\mu_{m}$-quotient of the above equation in general but 
comparing with \cite[6.8(vi)]{KSB} 
anyhow implies $m=1$. Now we decompose $ug(u,z)\in k[[x,y]]$ into the 
finite product of irreducible formal power series $g_{i}(u,z)$s for 
$i=1,\cdots,l$. We can and do suppose $g_{1}=u$. 

Now we consider the formal blow up of the formal scheme 
$${\rm Spf}(\widehat{\mathcal{O}_{p,\X_{\min}}}
\simeq k[[x,y,z,u]]/F(x,y,z,u))$$ along the ideal sheaf $(x,g_{i})$. 
Clearly, it is covered by 
a regular open subset and 
$${\rm Spf}(k[[x,y,z,u]]/\frac{F(x,y,z,u)}{u}),$$ 
where 
$\frac{F(x,y,z,u)}{u}=g(z,u).$ 
Repeating the same procedure for $g_{i}$s inductively, 
we obtain a formal $k$-scheme $\widehat{\tilde{\X}}$ 
properly mapping to 
${\rm Spf}(\widehat{\mathcal{O}_{p,\X_{\min}}})$ 
whose singularities are all formally isomorphic to 
${\rm Spf}(k[[x,y,z,u]]/xy-g_{i}(z,u))$, 
where one of the formal coordinate $x$ is replaced during 
each formal blow up. 
Note that ${\rm Spf}(k[[x,y,z,u]]/xy-g_{i}(z,u))$ is 
either smooth or the completion of isolated singularity at 
$(x,y,z,u)=(0,0,0,0)$ which restricts to A type rational double 
points along $u=0$ hyperplane: hence satisfying the desired 
property of Proposition \ref{Step1} at the {\it formal (and local)  level}. We do the same procedures at all $p\in {\rm Sing}(\X_{\min})$. 

Finally, using \cite[Theorem 3.2, \S5]{Art.modif}, 
\footnote{this allows any field or excellent Dedekind scheme as a base}
there is 
an algebraic $k$-space $\tilde{\X}\to \X$ which 
gives rise to the above $\widehat{\tilde{\X}}\to {\rm Spf}(\widehat{\mathcal{O}_{p,\X_{\min}}})$ for any $p\in \X_{\min}$. 
We complete the proof of Proposition \ref{Step1}. 
\end{proof}

\begin{Rem}\label{Qfac}
The above type \eqref{eqn} 
terminal singularity, with possibly nonunit $g(z,u)$ 
are not contained in the classification list of the 
singularities \cite[4.4]{kwmt} due to the 
assumption in {\it loc.cit} 1.1 (3) which excludes it. 
Indeed, 
if you take \'etale presentation of such algebraic space, 
and pull back the ideal we blow up, we observe that 
above type singularity violates the condition 1.1(3) of \cite{kwmt} 
(or only the $\Q$-Cartierness of $V_{i}$s in our setting). 
This arguments adds another 
explaination to the proof of \cite[4.4]{kwmt}. 
%(By the way, we remind that 
%\cite[4.8]{Maulik}, 
%\cite[\S 3]{Mat}, \cite[Proposition 
%3.1]{LM} depend on the classification.)
\end{Rem}

\begin{comment}
\begin{Rem}\label{missing}
We also just remark that $\X'''$ in \cite[3.1]{Mat} 
is often {\it not} regular nor terminal a priori 
(at double locus in the central fiber) 
due to the effect of base change. 
Proof of \cite[3.1]{LM} has the same minor inaccuracy. Indeed, along 
the double curve one obtains non-isolated cA type singularity 
and even worse at the point where three components intersect 
a priori, 
even if the concern of \cite{Mat, LM} is for potentially good reduction 
case. 
The author believes the papers work after small modifications. 
Indeed, even if the irregularity of the total space violates the discussion, 
one can take toroidal (log) crepant 
blow ups. Hence, only 
just a little additional arguments seem necessary a priori. 
We refer the interested readers to \cite{FS} or more general \cite[2.8]
{galaxy} for instance, 
which explain how to (partially) resolve it in a (log) crepant manner 
as corresponding to ``finer and finer subdivision'' of 
dual graph. 
\marginpar{(暫定的に半ば個人的メッセージを書いただけ)}

For the purpose of their papers, 
either by employing the same arguments as Theorem \ref{KPP.geom} 
or do toroidal (log) crepant resolution as in 
\cite{FS, galaxy} should resolve the problem I believe. 
\end{Rem}
\end{comment}

Now we come back to continue the proof of Theorem \ref{KPP.geom}. 
Our second step is the following, a wellknown 
procedure due to \cite{Brieskorn, Slodowy} 
(see also recent \cite{SB.simul0, SB.simul, Mat, LM}). 
We write here just for clarity and self-containedness. 

\begin{Lem}[Simultaneous resolution]\label{Step2}
Consider any proper algebraic space 
$\pi_{\Y}\colon \mathcal{Y}\to C$ over a smooth $k$-curve $C$ whose 
generic fiber is smooth with trivial canonical divisor, 
such that any closed 
point $y\in \Y$ satisfies one of the followings: 
\begin{enumerate}
\item $\Y$ is regular at $y$ and the fiber 
$\pi_{\Y}^{-1}(\pi_{\Y}(y))$ is normal crossing around $y$ or 
\item \label{rdp.fib} $\pi_{\Y}^{-1}(\pi_{\Y}(y)) \ni y$ is a rational double point. 
\end{enumerate}
Further, suppose that there is a birational 
strictly semistable model $\X\to C$ with the same generic fiber 
as that of $\Y\to C$. 

Then there is a simultaneous (small) resolution 
$\tilde{\X}\to \Y$ in the category of algebraic spaces. 
\end{Lem}

\begin{proof}[proof of Lemma \ref{Step2}]
Fix a rational double point $y$ of fiber as \eqref{rdp.fib}, 
take its neighborhood $U\subset \Y$ and 
denotes its complement as $Z$. 

The strict semistability of 
$\X^{o}:=\X \setminus \overline{(Z\setminus \pi_{\Y}^{-1}(y))}$ 
implies that action of 
the Galois group ${\rm Gal}(\overline{K}^{sep}/K)$ 
\footnote{not only the inertia group around $\pi_{\Y}(y)$. 
By the way, this difference becomes crucial at least 
over general CDVR 
cf., \cite[5.3]{Mat}, \cite[\S 7]{LM}} on 
$H^{2}_{\text{\'et}}(\X^{o}_{\overline{K(\eta)}}, \mathbb{Q}_{l})$ 
is unipotent (\cite{RZ}) and of finite index hence 
trivial action on it. 
Then, 
Brieskorn simultaneous resolution due to 
\cite{Brieskorn, Slodowy} 
and also \cite[Corollary 2.13]{SB.simul}, \cite[4.7]{SB.simul0} 
shows the existence of 
simultaneous resolution of $y$ without base change, 
at the category of algebraic spaces. 
Indeed, note that the assumption made at the first paragraph of 
\cite[\S 2]{SB.simul} automatically holds since we work over 
an algebraically closed field. 
We do the same for all 
fiberwise rational double points $y$. 
If $k=\C$, we could replace the use of 
above \cite{SB.simul} by more classical 
\cite{Landman}. 
\end{proof}
Combining Proposition \ref{Step1} and \ref{Step2}, we 
conclude the proof of Theorem \ref{KPP.geom}. 
\end{proof}

\subsection{For generalizations}\label{generalbase}

We discuss towards generalization of the above approach. 
Firstly, it is natural to expect to recover the analytic version 
Theorem \ref{KPP.original} 
fully from the MMP method. 

\begin{Rem}[Complex analytic version]
For more complex analytic statement Theorem \ref{KPP.original}, 
the above arguments with following slight verbatim modifications 
give an alternative partial proof when we further assume that 
$\X\to C$ is proper and projective away from 
singular fibers. 
\begin{enumerate}
\item \label{relMMP} For relative MMP, use \cite{KNX} (cf., also 
\cite{HP}) instead of the version with algebraic base 
\cite{Fjn.ss, HX, HW}, 
\item  use original 
\cite{Landman, Brieskorn, 
Slodowy} rather than \cite{SB.simul0, SB.simul} 
\item prove the analytic version of 
the Proposition \ref{Step1} by replacing our use of formal blow up 
by direct small analytic 
blow up  
\end{enumerate}
To recover Theorem \ref{KPP.original} fully by this method, 
we (only) need  relative/semistable extension of 
K\"ahler (absolute) MMP after \cite{HP}. 

On the other hand, also recall that Theorem \ref{KPP.original} 
can not be generalized to the case when the central fiber is allowed to contain non-K\"ahler components $V$, 
as the counterexamples found by Nishiguchi \cite[Theorem 4.4, \S 5 (compare with \S2)]{Nishi} show. 
In his counterexamples, $V$ are certain VII surfaces which he calls CB 
surfaces. 
\end{Rem}

%%%%%%%%%%%%%%%%%%%

We also review the following result 
due to \cite{Mat, LM} for convenience. 
That is, over more general Dedekind schemes with 
arbitrary (possibly mixed) characteristics and 
non-closed residue fields, 
at least we know the following version 
which allow finite base change. 
The reason we only discuss such weak version 
comes from the 
additional implicit assumptions made in the proofs of 
\cite[Corollary 2.13]{SB.simul}, \cite[4.7]{SB.simul0}. 
Below, we do not use \cite{SB.simul0, SB.simul}. 

\begin{Prop}[{\cite[3.1]{Mat}, \cite[3.1]{LM}}]
\label{mix}

Take 
any projective scheme 
over an excellent Dedekind scheme $C$, denoted as 
$\pi \colon \X\to C$, 
whose generic fiber has a trivial canonical divisor. 
We assume it has a strictly semistable model. 

Then, possibly after a finite base change of $C$, 
is a birational map from an 
algebraic space 
$\tilde{\X}\dashrightarrow \X$ such that 
\begin{enumerate}
\item $\tilde{\X}$ is 
strictly semistable over $C$ 
%\footnote{this means the fibers are all 
%{\it geometrically reduced} simple normal crossing divisors}
\item 
$K_{\tilde{\X}/C}\sim_{C} 0$. 
\end{enumerate}

\end{Prop}

\begin{proof}
This is essentially proven in the arguments during the 
proofs of 
\cite[3.1]{Mat} and \cite[3.1]{LM}, although the latter forgets 
the following process \eqref{toroidal.resol}. We just review that 
the differences with the above 
proof of Theorem \ref{KPP.geom} are: 
\begin{enumerate}
\item \label{toroidal.resol}
do natural toroidal resolution (cf., e.g., \cite[2.9.2]{Saito}) 
after a priori nontrivial 
base change, corresponding to a regular subdivision of a cone complex 
\item replace 
the partial 
use of \cite{SB.simul0, SB.simul} by more classical \cite{Art.sim} 
for simultaneous resolution (in the proof of Proposition \ref{Step2}), 
\item replace the use of \cite{Fjn.ss, HX, HW} by \cite{kwmt, TY, Bhattetc}). 
\end{enumerate}
\end{proof}

Again, it seems reasonable to expect refinement of above at the level 
of Theorem \ref{KPP.geom} may be possible by similar method (see 
\cite{SB.simul} again) 
but we do 
not pursue further in this notes. 

\begin{ack}
We would like to thank 
C.Birkar, O.Fujino, T.Koshikawa, Y.Matsumoto, S.Mori, 
T.Takamatsu, 
H.Tanaka, C.Xu and S.Yoshikawa for helpful comments. 

\end{ack}

%%%%%%%
%%%%%%%%
%%%%%%%
%%%%%%%

\bigskip

\footnotesize 
\noindent
Email address: \footnotesize 
{\tt yodaka@math.kyoto-u.ac.jp} \\
Affiliation: Department of Mathematics, Kyoto university, Japan  \\

\end{document}